\documentclass[12pt,fleqn,a4paper]{amsart}

\usepackage[utf8]{inputenc}
\usepackage[foot]{amsaddr}
\usepackage[pdftex]{lscape}
\usepackage{comment}
\usepackage[in]{fullpage}
\usepackage[colorlinks,breaklinks,linkcolor=blue,anchorcolor=blue,citecolor=blue]{hyperref}
\usepackage{amsfonts,amssymb,amsthm}
\usepackage[alphabetic,nobysame]{amsrefs}
\catcode`\@=11

\def\@settitle{%
  \baselineskip14\p@\relax
    {\Large\bfseries
  \@title}}

\def\@setauthors{%
  \begingroup
  \def\thanks{\protect\thanks@warning}%
  \trivlist
  \footnotesize \@topsep45\p@\relax
  \advance\@topsep by -\baselineskip
  \item\relax
  \author@andify\authors
  \def\\{\protect\linebreak}%
  {\sc\fontsize{12}{10}\selectfont\authors}%
  \ifx\@empty\contribs
  \else
    ,\penalty-3 \space \@setcontribs
    \@closetoccontribs
  \fi
  \endtrivlist
  \endgroup
}

\def\@secnumfont{\bfseries}%

\def\section{\@startsection{section}{1}%
  \z@{.7\linespacing\@plus\linespacing}{.5\linespacing}%
  {\normalfont\bf}}

\renewcommand{\BibLabel}{%
    \Hy@raisedlink{\hyper@anchorstart{cite.\CurrentBib}\hyper@anchorend}%
    [\thebib]\hfill%
}

\DefineSimpleKey{bib}{arxiv}

\BibSpec{article}{%
    +{}  {\PrintAuthors}                {author}
    +{,} { \textit}                     {title}
    +{.} { }                            {part}
    +{:} { \textit}                     {subtitle}
    +{,} { \PrintContributions}         {contribution}
    +{.} { \PrintPartials}              {partial}
    +{,} { }                            {journal}
    +{}  { \textbf}                     {volume}
    +{}  { \PrintDatePV}                {date}
    +{,} { \issuetext}                  {number}
    +{,} { \eprintpages}                {pages}
    +{,} { }                            {status}
    +{,} { \PrintDOI}                   {doi}
    +{,} { \arxiv}           {arxiv}
    +{}  { \parenthesize}               {language}
    +{}  { \PrintTranslation}           {translation}
    +{;} { \PrintReprint}               {reprint}
    +{.} { }                            {note}
    +{.} {}                             {transition}
    +{}  {\SentenceSpace \PrintReviews} {review}
}

\newcommand{\arxiv}[1]{\tt arxiv:\hspace{0pt}{\href{http://arxiv.org/abs/#1}{#1}}}

\catcode`\@=12

\usepackage{parskip}

\usepackage[cmtip,matrix,arrow]{xy} \SelectTips{cm}{10}
\newcommand{\cxymatrix}[1]{\vcenter{\xymatrix@=15pt{#1}}}
\usepackage{rotating}

\newcommand{\xysubseteq}{\ar@{}[r]|{\displaystyle\subseteq}}
\newcommand{\xysubseteqdown}{\ar@{}[d]|{\rotatebox{90}{$\supseteq$}}}

\usepackage{tikz}
\usetikzlibrary{calc}

\usepackage{aliascnt}

\newtheorem{theorem}{Theorem}[section]
\newaliascnt{lemma}{theorem}

\newtheorem{lemma}[lemma]{Lemma}
\aliascntresetthe{lemma}
\newaliascnt{corollary}{theorem}

\aliascntresetthe{corollary}
\newaliascnt{proposition}{theorem}

\aliascntresetthe{proposition}

\theoremstyle{definition}
\newaliascnt{definition}{theorem}

\newtheorem{definition}[definition]{Definition}
\aliascntresetthe{definition}

\newaliascnt{remark}{theorem}
\newtheorem{remark}[remark]{Remark}
\aliascntresetthe{remark}
\newtheorem{remarks}[remark]{Remarks}

\newtheorem*{remark*}{Remark}

\newaliascnt{example}{theorem}

\aliascntresetthe{example}

\usepackage{enumitem}
\setlist[enumerate,2]{label=\textit{\alph*)},ref=\textit{\alph*})}
\setlist[enumerate,1]{label=\textit{\roman*)},ref=\textit{\roman*})}

\usepackage{collcell}
\usepackage{array,longtable}
\newcolumntype{H}{>{\rule{0pt}{18pt}$}l<{$}}
\newcolumntype{L}{>{$}l<{$}}
\newcolumntype{R}{>{\collectcell\boldsymbol}r<{\endcollectcell}}
\def\ph{\relax}
\newcommand{\I}[1]{\IndexOben{\siv_#1}}
\newcommand{\Ig}[1]{\IndexOben{\gv_#1}}

\def\rlap#1{\hbox to 0pt{$\scriptstyle#1$\hss}}

\usepackage[capitalize]{cleveref}
\usepackage{microtype}

\def\wss.{weak spherical datum}
\def\wsss.{weak spherical data}

\newcommand{\fg}{\mathfrak{g}}
\newcommand{\fl}{\mathfrak{l}}
\newcommand{\fs}{\mathfrak{s}}
\newcommand{\ft}{\mathfrak{t}}
\renewcommand{\sl}{\fs\fl}

\newcommand{\Xq}{{\overline X}}
\newcommand{\Yq}{{\overline Y}}
\newcommand{\cN}{{\mathcal N}}

\newcommand{\CC}{\mathbb{C}}

\newcommand{\QQ}{\mathbb{Q}}
\newcommand{\ZZ}{\mathbb{Z}}

\renewcommand{\rho}{\varrho}
\renewcommand{\phi}{\varphi}
\renewcommand{\epsilon}{\varepsilon}

\newcommand{\leer}{\varnothing}
\newcommand{\<}{\langle} 
\renewcommand{\>}{\rangle}
\newcommand{\into}{\hookrightarrow}
\newcommand{\auf}{\twoheadrightarrow}
\renewcommand{\[}{\begin{equation}}
\renewcommand{\]}{\end{equation}}

\DeclareMathOperator{\rk}{rk}
\DeclareMathOperator{\Hom}{Hom}
\DeclareMathOperator{\Ad}{Ad}
\DeclareMathOperator{\res}{res}

\newcommand{\G}{{\mathbf{G}}}
\newcommand{\cD}{\mathcal{D}}
\newcommand{\cF}{\mathcal{F}}
\newcommand{\cQ}{\mathcal{Q}}
\newcommand{\cZ}{\mathcal{Z}}

\newcommand{\PGL}{\mathrm{PGL}}
\newcommand{\PSp}{\mathrm{PSp}}
\newcommand{\PSO}{\mathrm{PSO}}
\newcommand{\SL}{\mathrm{SL}}
\newcommand{\SO}{\mathrm{SO}}
\newcommand{\Sp}{\mathrm{Sp}}
\newcommand{\Spin}{\mathrm{Spin}}

\newcommand{\sA}{\mathsf{A}}
\newcommand{\sB}{\mathsf{B}}
\newcommand{\sC}{\mathsf{C}}
\newcommand{\sD}{\mathsf{D}}
\newcommand{\sF}{\mathsf{F}}
\newcommand{\sG}{\mathsf{G}}

\renewcommand{\a}{\alpha}
\renewcommand{\b}{\beta}
\renewcommand{\d}{\delta}
\newcommand\e{\varepsilon}
\newcommand\g{\gamma}
\newcommand{\av}{\a^\vee}
\newcommand{\bv}{\b^\vee}
\newcommand{\dv}{\d^\vee}
\newcommand{\gv}{\g^\vee}
\newcommand{\Gv}{G^\vee}
\newcommand{\Lv}{L^\vee}
\newcommand{\Tv}{T^\vee}
\newcommand{\ftv}{\ft^\vee}
\newcommand{\Av}{A^\vee}

\newcommand{\tv}{\tau^\vee}
\newcommand{\Phiv}{\Phi^\vee}
\newcommand{\fgv}{\fg^\vee}
\newcommand{\flv}{\fl^\vee}
\newcommand{\Ga}{G\ass}
\newcommand{\Ta}{T\ass}

\newcommand{\Siv}{\Sigma^\vee}
\newcommand{\siv}{\sigma^\vee}
\newcommand{\Sa}{\Sigma\ass}

\newcommand{\ass}{^\wedge}
\newcommand{\p}{^p}
\newcommand{\+}{+\ldots+}
\renewcommand{\tilde}{\widetilde}
\newcommand{\XiS}{\tilde\Xi}
\newcommand{\XiSv}{\tilde\Xi^\vee}

\newcommand{\bild}[1]{\begin{tikzpicture}[scale=\scalar,baseline]
\def\Punkt{(0,0)}\x#1\end{tikzpicture}}

\def\ph{\vrule width 0pt height 8pt depth 2pt}

\def\Arechts{%
\draw \Punkt --++(1,0);
\edef\Punkt{($\Punkt+(1,0)$)}
\filldraw \Punkt circle (2pt);}

\def\Alinks{%
\filldraw \Punkt circle (2pt);}

\def\Brechts{%
\edef\PUNKT{($\Punkt+(0.05,0.04)$)}
\draw \PUNKT --++(0.8,0);
\edef\PUNKT{($\Punkt+(0.05,-0.04)$)}
\draw \PUNKT --++(0.8,0);
\edef\PUNKT{($\Punkt+(0.7,0.15)$)}
\draw\PUNKT--++(0.2,-0.15)--++(-0.2,-0.15);
\edef\Punkt{($\Punkt+(1,0)$)}
\filldraw \Punkt circle (2pt);}

\def\Crechts{%
\edef\PUNKT{($\Punkt+(0.15,0.04)$)}
\draw \PUNKT--++(0.8,0);
\edef\PUNKT{($\Punkt+(0.15,-0.04)$)}
\draw \PUNKT--++(0.8,0);
\edef\PUNKT{($\Punkt+(0.3,0.15)$)}
\draw\PUNKT--++(-0.2,-0.15)--++(0.2,-0.15);
\edef\Punkt{($\Punkt+(1,0)$)}
\filldraw \Punkt circle (2pt);}

\def\Drechts{%
\edef\PUNKT{($\Punkt+(0.5,0.5)$)}
\filldraw \PUNKT circle (2pt);
\edef\PUNKT{($\Punkt+(0.5,-0.5)$)}
\filldraw \PUNKT circle (2pt);
\draw \Punkt--++(0.5, 0.5);
\draw \Punkt--++(0.5,-0.5);}

\def\Glinks{%
\edef\PUNKT{($\Punkt+(-0.2,0.07)$)}
\draw \PUNKT--++(-0.8,0);
\edef\PUNKT{($\Punkt+(-0.2,-0.07)$)}
\draw \PUNKT--++(-0.8,0);
\edef\PUNKT{($\Punkt+(-0.1,0)$)}
\draw \PUNKT--++(-0.9,0);
\edef\PUNKT{($\Punkt+(-0.3,0.15)$)}
\draw\PUNKT--++(0.2,-0.15)--++(-0.2,-0.15);
\filldraw \Punkt circle (2pt);
}

\def\Grechts{%
\edef\PUNKT{($\Punkt+(0.2,0.07)$)}
\draw \PUNKT--++(0.8,0);
\edef\PUNKT{($\Punkt+(0.2,-0.07)$)}
\draw \PUNKT--++(0.8,0);
\edef\PUNKT{($\Punkt+(0.1,0)$)}
\draw \PUNKT--++(0.9,0);
\edef\PUNKT{($\Punkt+(0.3,0.15)$)}
\draw\PUNKT--++(-0.2,-0.15)--++(0.2,-0.15);
\edef\Punkt{($\Punkt+(1,0)$)}
\filldraw \Punkt circle (2pt);}

\def\ddd{
\edef\PUNKTa{($\Punkt+(0.08,0)$)}
\edef\PUNKTb{($\Punkt+(0.6,0)$)}
\draw \PUNKTa -- \PUNKTb;
\edef\PUNKTa{($\Punkt+(0.8,0)$)}
\filldraw \PUNKTa circle (0.3pt);
\edef\PUNKTa{($\Punkt+(1,0)$)}
\filldraw \PUNKTa circle (0.3pt);
\edef\PUNKTa{($\Punkt+(1.2,0)$)}
\filldraw \PUNKTa circle (0.3pt);
\edef\PUNKTa{($\Punkt+(1.4,0)$)}
\edef\PUNKTb{($\Punkt+(1.92,0)$)}
\draw \PUNKTa -- \PUNKTb;
\edef\PUNKTa{($\Punkt+(2,0)$)}
\filldraw \PUNKTa circle (2pt);
\edef\Punkt{($\Punkt+(2,0)$)}
}

\def\IndexRechts#1{%
\draw \Punkt node[right]{$\scriptstyle#1$};}

\def\IndexLinksex(#1,#2;#3){%
\draw (#1,#2) node[left]{$\scriptstyle#3$};}

\def\IndexOben#1{%
\draw \Punkt node[above]{$\ph\scriptstyle#1$};}

\def\IndexObenex(#1,#2;#3){%
\draw (#1,#2) node[above]{$\ph\scriptstyle#3$};}

\def\IndexUnten#1{%
\draw \Punkt node[below]{$\ph\scriptstyle#1$};}

\def\IndexUntenex(#1,#2;#3){%
\draw (#1,#2) node[below]{$\ph\scriptstyle#3$};}

\def\FarbeS[#1,#2]{%
\filldraw \Punkt node[above]{$\ph\scriptstyle#1$} node[below]{$\ph\scriptstyle#2$} circle (5pt);}

\def\FarbeSex(#1,#2;#3,#4){%
\filldraw (#1,#2) node[above]{$\ph\scriptstyle#3$} node[below]{$\ph\scriptstyle#4$} circle (5pt);}

\def\FarbeU{%
\edef\PUNKT{($\Punkt+(0,-0.4)$)}
\draw \PUNKT circle (5pt);}

\def\FarbeUex(#1){%
\edef\PUNKT{($(#1)+(0,-0.4)$)}
\draw \PUNKT circle (5pt);}

\def\FarbeO{%
\edef\PUNKT{($\Punkt+(0,0.4)$)}
\draw \PUNKT circle (5pt);}

\def\FarbeOex(#1){%
\edef\PUNKT{($(#1)+(0,0.4)$)}
\draw \PUNKT circle (5pt);}

\def\FarbeM{%
\draw \Punkt circle (5pt);}

\def\FarbeD{%
\filldraw \Punkt circle (5pt);}

\def\FarbeMex(#1){%
\draw (#1) circle (5pt);}

\def\lzz(#1,#2){%
\draw[thick,rotate around ={135:(#1,#2)}]
(#1,#2)--++(-0.1,0.1)
foreach \x in {1,...,4} {--++(-0.05,0.05)--++(-0.05,-0.05)}
--++(-0.06,0.06);}

\def\rzz(#1,#2){%
\draw[thick,rotate around ={45:(#1,#2)}]
(#1,#2)--++(0.1,0.1)
foreach \x in {1,...,4} {--++(0.05,0.05)--++(0.05,-0.05)}
--++(0.06,0.06);}

\def\zick(#1,#2){\draw[thick](#1,#2)--++(0.2,0.2);}
\def\zack(#1,#2){\draw[thick]($(#1,#2)+(-0.2,0.2)$)--(#1,#2);}

\def\dreizz(#1,#2){%
\draw[thick]
($(#1,#2)+(0.2,0)$)
foreach \x in {1,...,5} {--++(0.05,0.05)--++(0.05,-0.05)}
--++(0.05,0.05)
--++(0.051,-0.051);
}

\def\zweizz(#1,#2){%
\draw[thick]
($(#1,#2)+(0.2,0)$)
foreach \x in {1,...,3} {--++(0.05,0.05)--++(0.05,-0.05)}
--++(0.05,0.05)
--++(0.051,-0.051);
}

\def\vierzz(#1,#2){%
\draw[thick]
($(#1,#2)+(0.1,0)$)
foreach \x in {1,...,6} {--++(0.05,0.05)--++(0.05,-0.05)}
--++(0.05,0.05)
--++(0.051,-0.051);
}

\def\fuenfzz(#1,#2){%
\draw[thick]
($(#1,#2)+(0.2,0)$)
foreach \x in {1,...,9} {--++(0.05,0.05)--++(0.05,-0.05)}
--++(0.05,0.05)
--++(0.051,-0.051)
;
}

\def\zz{
\draw[thick]\Punkt--++(0.2,0.2)
foreach \x in {1,...,6} {--++(0.05,0.05)--++(0.05,-0.05)}
--++(0.2,-0.2);
}

\def\WurzelAn{\FarbeM\draw[thick]\Punkt--++(0.2,0.2)
foreach \x in {1,...,36} {--++(0.05,0.05)--++(0.05,-0.05)}
--++(0.2,-0.2);\Arechts\ddd\Arechts\FarbeM}

\def\scalar{0.5}
\def\ph{\relax}

\newcommand{\x}{\relax}

\title[]{Functoriality properties of the dual group}

\author[]{Friedrich Knop}
\address[]{Dept. Mathematik\\FAU Erlangen-Nürnberg\\
Cauerstraße 11\\
D-91058 Erlangen}

\subjclass[2010]{17B22, 14L30, 11F70} \keywords{Spherical variety,
  Langlands dual group, root system, algebraic group, reductive
  group}

\begin{document}

\begin{abstract}

  Let $G$ be a connected reductive group. Previously, is was shown
  that for any $G$-variety $X$ one can define a the dual group
  $G^\vee_X$ which admits a natural homomorphism with finite kernel to
  the Langlands dual group $G^\vee$ of $G$. Here, we prove that the
  dual group is functorial in the following sense: if there is a
  dominant $G$-morphism $X\to Y$ or an injective $G$-morphism $Y\to X$
  then there is a unique homomorphism with finite kernel
  $G^\vee_Y\to G^\vee_X$ which is compatible with the homomorphisms to
  $G^\vee$.

\end{abstract}

\maketitle         

\section{Introduction}

Let $G$ be a connected reductive group defined over an algebraically
closed field $k$ of characteristic zero. To any $G$-variety $X$ one
can attach a finite reflection group $W(X)$ (its ``little Weyl
group'') which, loosely speaking, determines the large scale geometry
of $X$ (see Brion \cite{Brion} and \cite{KnopAB}).

While it is known that $W(X)$ is a subgroup of the Weyl group of $G$,
it is, in general, not true that it is the Weyl group of some subgroup
of $G$. But surprisingly, the Langlands dual group $\Gv$ of $G$ does
contain such a subgroup.

At least in the case when $X$ is spherical, this was first hinted at
in work of Gaitsgory and Nadler, \cite{GaitsgoryNadler}, who
constructed a reductive subgroup of $\Gv$ whose Weyl group is most
likely equal to $W(X)$. Later Sakellaridis and Venkatesh, \cite{SV},
refined (at least for $X$ spherical) the description of a hypothetical
subgroup with Weyl group $W(X)$. In particular, they worked out
precisely how it should embed into $\Gv$. They also replaced the
subgroup by a particular finite cover $\Gv_X$, the \emph{dual group of
  $X$}, which carries more information about $X$.

In \cite{KnopSchalke}, it was shown that the Sakellaridis-Venkatesh
construction does indeed work, i.e., that there is a homomorphism
$\phi_X:\Gv_X\to\Gv$ as predicted in \cite{SV}. The approach of
\cite{KnopSchalke} is purely combinatorial.

In the present paper we investigate the question whether the
assignment $X\mapsto (\Gv_X,\phi_X)$ can be turned into a functor. To
this end, we are going to normalize the homomorphism $\phi_X$ in such
a way that it becomes unique up to conjugation by an element of the
maximal torus of $\Gv_X$. The main result of the present paper is:

\begin{theorem}\label{thm:MainTheorem}

  Let $X$ and $Y$ be two $G$-varieties. Assume that there is either a
  dominant $G$-morphism $f:X\to Y$ or a generically injective
  $G$-morphism $Y\to X$. Then there exists a unique homomorphism
  (necessarily with finite kernel) $\eta:\Gv_Y\to\Gv_X$ such that
  $\phi_Y=\phi_X\circ\eta$.

\end{theorem}

In the body of the paper, we prove a more precise version of the
theorem (see Theorems \ref{thm:main2} and \ref{thm:main3}).

The proof of \cref{thm:MainTheorem} proceeds in several steps: first
we treat the case of a dominant morphism. First, the theorem is
reduced to the case when both $X$ and $Y$ are homogeneous with $Y$
being of rank $1$ and $f$ being proper. Then we use a classification
(due to Akhiezer \cite{Akhiezer} and Panyushev
\cite{PanyushevRankOne}) to check the assertion case-by-case. To this
end, we determine, given a spherical $G$-variety $G/H$ of rank $1$,
the Luna data of $G/P$ where $P$ runs through all maximal parabolic
subgroups of $H$. This might be of independent interest since the
morphisms $G/P\to G/H$ are in a sense minimal among all dominant
$G$-morphisms. The case of injective morphisms will finally follow
from the dominant one.

As opposed to \cite{KnopSchalke} we are going to argue much more
geometrically than combinatorially. This is is due to the fact that
the the \wsss. used in \cite{KnopSchalke} do not possess sufficient
functorial properties.

\section{The dual group and distinguished homomorphisms}
\label{sec:dualgroup}

Let $G$ be a connected reductive group defined over an algebraically
closed ground field $k$ of characteristic $0$. Let $B\subseteq G$ be a
Borel subgroup and $T\subseteq B$ a maximal torus. Let
$\Lambda:=\Xi(B)$ be the weight lattice, $\Phi\subset\Lambda$ the root
system of $G$, and $S\subseteq\Phi$ the set of simple roots with
respect to $B$.

We recall the dual group $\Gv_X$ of a $G$-variety $X$. A rational
function $f\in k(X)$ is $B$-semiinvariant with character
$\chi_f\in\Lambda$ if $f(b^{-1}x)=\chi_f(b)f(x)$ for all $b\in B$ and
$x\in X$ where both sides are defined. All characters $\chi_f$ form a
subgroup $\Xi=\Xi(X)$ of $\Lambda$, the \emph{weight lattice of
  $X$}. The rank of $\Xi(X)$ is called the \emph{rank of $X$} and is
denoted by $\rk X$.

Now consider a discrete valuation $v:k(X)\to\QQ\cup\{\infty\}$. It is
called \emph{central} if it is $G$-invariant and restricts to the
trivial valuation on the field $k(X)^B$ of rational
$B$-invariants. Then $v(f)$ depends, for any $B$-semiinvariant $f$,
only on its character $\chi_f$. Thus we get a map
\[
  \rho:\cZ(X)\to\cN(X):=\Hom(\Xi,\QQ)
\]
where $\cZ(X)$ is the set of all central valuations. It was proven in
\cite{LunaVust} that $\rho$ is injective. Hence we may and will
identify $\cZ(X)$ with a subset of the $\QQ$-vector space $\cN(X)$.

One can show that $\cZ(X)$ is a finitely generated convex cone which
is not contained in a hyperplane. Let
\[
  \Sigma=\Sigma(X)=\{\sigma_1,\ldots,\sigma_s\}\subseteq\Xi_\QQ:=\Xi\otimes\QQ
\]
be a minimal set of outward normal vectors (so-called \emph{spherical
  roots of $X$}) such that
\[
  \cZ(X)=\{a\in\cN(X)\mid a(\sigma_1)\le0,\ldots,a(\sigma_s)\le0\}.
\]
The $\sigma_i$ are only unique up to positive factors and there are
several normalizations possible. The one which we are adopting uses
the fact that each $\sigma_i$ lies in the intersection
$\Xi_\QQ\cap\QQ S$. Thus we can and will normalize $\sigma_i$ is such
a way that it is primitive in the root lattice $\ZZ S$. Therefore,
every $\sigma_i$ is a linear combination
$\sum_{\alpha\in S}n_\alpha\alpha$ with integral coprime coefficients
which one can show to be non-negative. The \emph{support} $|\sigma_i|$
of $\sigma_i$ is the set $\{\alpha\in S\mid n_\alpha>0\}$. More
generally, we put
$|\Sigma_0|=\cup_{\sigma\in\Sigma_0}|\sigma|\subseteq S$ for any
subset $\Sigma_0\subseteq\ZZ S$.

A third invariant of $X$ is a certain set $S\p=S\p(X)\subseteq S$ of
simple roots. It consists of all $\alpha\in S$ (called \emph{parabolic
  for $X$}) such that $P_\alpha x=Bx$ for generic $x\in X$. Here
$P_\alpha\subseteq G$ is the minimal parabolic subgroup corresponding
to $\alpha$. In other words, the parabolic subgroup $Q(X)$
corresponding to $S\p$ is the stabilizer of a generic $B$-orbit.

The coefficients $n_\alpha$ are always non-negative. In fact much more
is true. One can show that the triple
$(|\sigma|,\sigma,S\p\cap|\sigma|)$ will always appear in
\cref{tab:NSR}. The items correspond to spherical varieties of rank
$1$ (listed in \cref{tab:rankone}) which will be explained in more
detail in \cref{sec:rankone}.

\begin{table}[!h]
  \caption{}\label{tab:1}
  \label{tab:NSR}
  \begin{tabular}{LLL}

    |\sigma|&\sigma&S\p\cap|\sigma|\\

    \noalign{\smallskip\hrule\smallskip}
    \sA_1&\a_1&\leer\\
    \sA_n,\ n\ge2&\a_1\+\a_n&\{\a_2,\ldots,\a_{n-1}\}\\
    \sB_n,\ n\ge2&\a_1\+\a_n&\{\a_2,\ldots,\a_n\}\\
    \sB_n,\ n\ge2&\a_1\+\a_n&\{\a_2,\ldots,\a_{n-1}\}\\
    \sC_n,\ n\ge3&\a_1+2\a_2\+2\a_{n-1}+\a_n&\{\a_1,\a_3,\ldots,\a_n\}\\
    \sC_n,\ n\ge3&\a_1+2\a_2\+2\a_{n-1}+\a_n&\{\a_3,\ldots,\a_n\}\\
    \sF_4&\a_1+2\a_2+3\a_3+2\a_4&\{\a_1,\a_2,\a_3\}\\
    \sG_2&2\a_1+\a_2&\{\a_2\}\\
    \sG_2&\a_1+\a_2&\{\a_1,\a_2\}\\

    \noalign{\smallskip\hrule\smallskip}

    \sD_2&\a_1+\a_2&\leer\\
    \sD_n,\ n\ge3&2\a_1\+2\a_{n-2}+\a_{n-1}+\a_n&\{\a_2,\ldots,\a_n\}\\
    \sB_3&\a_1+2\a_2+3\a_3&\{\a_1,\a_2\}\\

    \noalign{\smallskip\hrule\smallskip}

  \end{tabular}
\end{table}

One unfortunate feature of the normalization of spherical roots is the
possibility of $\Sigma\not\subseteq\Xi$. Therefore, we define the
modified weight lattice of $X$ as
\[\label{eq:XiS}
  \XiS=\XiS(X):=\Xi(X)+\ZZ\Sigma(X).
\]
According to \cite{KnopSchalke}*{Prop.~5.4}, the triple
$(\XiS,\Sigma,S\p)$ is a \emph{weak spherical datum}, i.e., satisfies:

\begin{itemize}

\item $\<\XiS\mid\av\>=0$ for all $\a\in S\p$.

\item $\<\XiS\mid\av-\bv\>=0$ whenever $\sigma=\a+\b\in\Sigma$ is of
  type $\sD_2$.

\item $\<\b\mid\av\>\ne-1$ whenever $\a,\b\in S$ with
  $\a,\a+\b\in\Sigma$.
    
\end{itemize}

Looking at \cref{tab:NSR} one realizes that there are two types of
spherical roots namely those which are also roots of $G$ and those
which are not. These types are separated by the middle horizontal
line. Each non-root $\sigma$ is the sum of two strongly orthogonal
roots $\gamma_1,\gamma_2$ as can be seen by inspection of
\cref{tab:assocroots}.  The set $\{\gamma_1,\gamma_2\}$ can be made
unique by requiring that
\[
  \gv_1-\gv_2=\dv_1-\dv_2\text{ with }\d_1,\d_2\in S.
\]
\begin{table}[!h]
  \caption{}
  \label{tab:assocroots}
  \begin{tabular}{LLLL}
    |\sigma|&\g_1,\g_2&\gv_1,\gv_2&\dv_1,\dv_2\\
    \hline
    \sD_2&\a_1,\ \a_2&\av_1,\ \av_2&\av_1,\ \av_2\\
    \hline
    \sD_n, n\ge3&(\a_1\+\a_{n-2})+\a_{n-1},&(\av_1\+\av_{n-2})+\av_{n-1},&\av_{n-1},\ \av_n\\
            &(\a_1\+\a_{n-2})+\a_n&(\av_1\+\av_{n-2})+\av_n\\
    \hline
    \sB_3&\a_1+\a_2+2\a_3,\ \a_2+\a_3&\av_1+\av_2+\av_3,\ 2\av_2+\av_3&\av_1,\ \av_2\\
    \hline
  \end{tabular}
\end{table}
It then follows that the restrictions of $\gv_1$ and $\gv_2$ to $\XiS$
coincide. Thus they define an element of $\XiSv:=\Hom(\XiS,\ZZ)$ which
is denoted by $\siv$. On the other hand, if $\sigma\in\Phi$ then the
coroot $\siv$ already has a meaning. Let
$\Siv:=\{\siv\mid\sigma\in\Sigma\}$. A fundamental fact about
\wsss. is the following

\begin{theorem}[\cite{KnopSchalke}*{Thm.~7.1}]

  Let $(\XiS,\Sigma,S\p)$ be a \wss.. Then $(\XiS,\Sigma,\XiSv,\Siv)$
  is a based root datum.

\end{theorem}

This theorem gives rise to the following definition.

\begin{definition}

  The dual group of a $G$-variety $X$ is the connected \emph{complex}
  reductive group $\Gv_X$ whose based root datum is the dual root
  datum $(\XiSv,\Siv,\XiS,\Sigma)$.

\end{definition}

\begin{remarks}

  \textit{i)} The Weyl group of $\Gv_X$ is, almost by definition,
  equal to the little Weyl group $W(X)$ of $X$. Observe that, due to
  our normalization, $\Sigma(X)$ and $W(X)$ determine each other
  unlike, e.g., the normalization used in \cite{KnopAuto} where the
  set of spherical roots carries additionally information about the
  automorphism group of $X$.

  \textit{ii)} The normalization of the spherical roots by being
  primitive in $\ZZ S$ is forced on us by the requirement that $\Gv_X$
  should map to $\Gv$ with finite kernel (see \cref{thm:schalke}
  below). This in turn forces the extension \eqref{eq:XiS} of
  character groups. Note, however, that for the representation
  theoretic purposes of \cite{SV} this is the wrong lattice since it
  yields multiplicities which are too big.

  \textit{iii)} In the Langlands program, the most common approach is
  to define the dual group only over $\CC$ and we follow this
  tradition. Working also simplifies some definitions and arguments,
  most notably \cref{def:distinguished} of a distinguished
  homomorphism in Lie algebraic terms. Nevertheless, it should be
  remarked that $\Gv_X$ can be defined over $\ZZ$ and that
  distinguished homomorphism exist over $\ZZ[\frac12]$ (see
  \cite{KnopSchalke}*{Prop.~11.1}). Also our main
  \cref{thm:MainTheorem} holds in that generality.

\end{remarks}

The dual group of $G$, i.e., the connected complex reductive group
whose root datum is dual to that of $G$ is denoted by $\Gv$. It is
equipped with a pinning, i.e., a choice of generating root vectors
$e_{\av}\in\fgv_{\av}$ with $\a\in S$.

It was proved in \cite{KnopSchalke} that there exists an almost
canonical homomorphism $\phi:\Gv_X\to\Gv$ with finite kernel. To make
this more precise, we define for each $\sigma\in\Sigma(X)$ a
one-dimensional subspace $\fgv_{\siv}$ of $\fgv$ as follows:
\[\label{eq:defsv}
  \fgv_{\siv}:=
  \begin{cases}
    \fgv_{\siv}&\text{if }\sigma\in\Phi,\\
    [\fgv_{\bv},e_{\dv_1}-e_{\dv_2}]&\text{if $\sigma$ is of type $\sD_{n\ge3}$,}\\
    [\fgv_{\bv},2e_{\dv_1}-e_{\dv_2}]&\text{if $\sigma$ is of type $\sB_3$,}\\
    \CC(e_{\dv_1}-e_{\dv_2})&\text{if $\sigma$ is of type $\sD_2$.}\\
  \end{cases}
\]
Here $\bv:=\gv_1-\dv_1=\gv_2-\dv_2$ in case $\sigma\not\in\Phi$. It is
easy to check that $\bv\in\Phiv$ unless $\sigma$ is of type $\sD_2$
when $\bv=0$. The definition implies that
\[
  \fgv_{\siv}\subseteq\fgv_{\gv_1}\oplus\fgv_{\gv_2}\subseteq\fgv.
\]

Next observe that the maximal tori $\Tv\subseteq\Gv$ and
$\Av_X\subseteq\Gv_X$ have the cocharacter group $\Lambda$ and
$\XiS(X)$, respectively. Therefore, the inclusion
$\XiS(X)\into\Lambda$ induces a homomorphism $\phi_A:\Av_X\to\Tv$ with
finite kernel.

\begin{definition}\label{def:distinguished}

  A homomorphism $\phi:\Gv_X\to\Gv$ is called \emph{distinguished} if
  $\res_{\Av_X}\phi=\phi_A$ and $\phi(\fgv_{X,\siv})=\fgv_{\siv}$ for
  all $\sigma\in\Sigma(X)$.

\end{definition}

Here is an immediate consequence of the main result of
\cite{KnopSchalke}:

\begin{theorem}\label{thm:schalke}

  Let $X$ be a $G$-variety. Then:

  \begin{enumerate}

  \item There exists a distinguished homomorphism
    $\phi_X:\Gv_X\to\Gv$.

  \item Any other distinguished homomorphism is of the form
    $\phi\circ\Ad(a)$ with $a\in\Av_X$.

  \item The kernel of $\phi_X$ is finite.

  \item The image $G^*_X:=\phi_X(\Gv_X)$ is a well-defined subgroup of
    $\Gv$, i.e., it is independent of the choice of $\phi_X$.

  \end{enumerate}

\end{theorem}

\begin{proof}

  \cite{KnopSchalke}*{Thm.~7.7} shows the existence of an adapted
  homomorphism $\phi:\Gv_X\to\Gv$ which means that $\fgv_{X,\siv}$ is
  mapped just diagonally into
  $\fgv_{\gv_1}\oplus\fgv_{\gv_2}\subseteq\fgv$ in case
  $\sigma\not\in\Phi$. More precisely, the image of $\phi$ is
  contained in the associated group $\Ga_X\subseteq\Gv$ (see
  loc.cit. Def.~7.2 and Thm 7.3). Thus, there an element $t$ of
  $\Ta_{\rm ad}$, the maximal torus of the adjoint group of $\Ga_X$,
  such that $\Ad(t)\circ\phi$ is distinguished (cf. loc.cit
  Thm.~7.10). The other parts follow from the construction of
  $\phi_X$.
\end{proof}

\begin{remarks}

  \emph{i)} Let $\Lv_X\subseteq\Gv$ be the Levi subgroup corresponding
  to $S\p(X)\subseteq S$. The pinning of $\Gv$ induces a pinning of
  $\Lv_X$. This in turn gives rise to a canonical principal
  homomorphism $\psi:\SL(2,\CC)\to\Lv_X$. Then it was shown,
  \cite{KnopSchalke}*{Prop.~9.10}, that the images of $\phi_X$ and
  $\psi$ commute with each other, i.e., they combine to a group
  homomorphism $\Gv_X\times SL(2)\to\Gv$. In fact, the normalization
  \eqref{eq:defsv} for $\sigma$ of type $\cD_{n\ge3}$ or $\sB_3$ is
  equivalent to this commutation property.

  \emph{ii)} Distinguished homomorphisms are invariant under certain
  automorphisms of $G$. More precisely, let $E$ be a group of
  automorphisms of the based root datum of $G$. Then $E$ acts
  canonically on $\Gv$ by fixing the chosen pinning $\{e_{\av}\}$. We
  say that \emph{$E$ and $X$ are compatible} if $E$ fixes $\Xi(X)$,
  $\Sigma(X)$, and $S\p(X)$. Then \eqref{eq:defsv} implies
  \[\label{eq:invariance}
    {}^s\fgv_{\siv}=\fgv_{{}^s\siv}\text{ for all $s\in E$ and
      $\sigma\in\Sigma(X)$.}
  \]
  This follows from \eqref{eq:defsv} together with the observation
  that ${}^s\dv_i=\overline\d^\vee_i$ in case $\sigma$ and
  $\overline\sigma={}^s\sigma$ are both of type $\sB_3$. Now
  \eqref{eq:invariance} implies that $E$ fixes $G^*_X$. Moreover, the
  $E$-action lifts uniquely to $\Gv_X$ such that $\phi_X$ is
  $E$-equivariant. Observe, though, that $E$ will in general
  \emph{not} fix any pinning of $\Gv_X$, i.e., the action may be
  \emph{non-standard} in the sense of \cite{KnopSchalke}*{\S10}.

  A typical situation we have in mind is if $G$ and $X$ are defined
  over a subfield $k_0\subseteq k$. Then the Galois group $E$ of $k_0$
  acts on the based root datum of $G$ by means of the so-called
  $*$-action. Since $X$ is defined over $k_0$ it is known (see
  \cite{KK}) that $E$ and $X$ are compatible.

  \emph{iii)} The normalization \eqref{eq:defsv} also plays a role in
  the proof of \cref{thm:main2} below. More precisely, it is needed to
  prove equation \eqref{eq:Dn}.

\end{remarks}

Now we come to homomorphisms between different dual groups. For this
let $X$, $Y$ be two $G$-varieties and let $\phi_X$, $\phi_Y$ be
distinguished homomorphisms.  A homomorphism $\eta:\Gv_Y\to\Gv_X$ is
called \emph{distinguished} if $\phi_Y=\phi_X\circ\eta$. Since
$\phi_X$ and $\phi_Y$ have finite kernel, $\eta$ is unique with finite
kernel if it exists. Here is the main result of the paper:

\begin{theorem}\label{thm:main2}

  Let $\phi:X\to Y$ be a dominant $G$-morphism between two
  $G$-varieties. Then there exists a distinguished homomorphism
  $\eta:\Gv_Y\to\Gv_X$. This implies, in particular, that
  $G^*_Y\subseteq G^*_X\subseteq\Gv$.

\end{theorem}

There is an analogous statement for injective morphisms. It is an easy
consequence of \cref{thm:main2} (see the proof following
\cref{thm:comparison}).

\begin{theorem}\label{thm:main3}

  Let $\phi:Y\to X$ be an injective $G$-morphism between two
  $G$-varieties (e.g., $Y$ is a $G$-stable subvariety of $X$). Then
  there exists a distinguished homomorphism $\eta:\Gv_Y\to\Gv_X$ and
  therefore, in particular, $G^*_Y\subseteq G^*_X\subseteq\Gv$.

\end{theorem}

The proof of \cref{thm:main2} will occupy the remainder of this paper.

\begin{remark}

  In principle, all statements can be formulated and should be valid
  in some form also over fields of positive characteristic
  $p$. However, the necessary changes would come at the expense of the
  readability of the paper so that we decided to treat the
  characteristic $0$ case separately. The main problems in positive
  characteristic are: First, the list of spherical roots in
  \cref{tab:1} has to be extended by roots obtained by inseparable
  isogenies. In particular, the $\sD_2$-roots $\a_1+p^n\a_2$ cause
  trouble. Secondly, the weight lattice $\Xi(X)$ may not be
  $W(X)$-stable, so has to be modified. Finally, our reasoning in
  \cref{sec:appendix} uses the classification of spherical
  varieties. This is more a matter of convenience but it would require
  considerable effort to work around it.

\end{remark}

\section{Reduction to rank one}

We start the proof of \cref{thm:main2} by a number of reduction
steps. Let $G'_X:=(G^*_X)'$ be the semisimple part of $G^*_X$. Observe
that $G'_X$ depends only on $\Sigma(X)$ and not on the lattice
$\Xi(X)$. Since the valuation cone $\cZ(X)$ is a birational invariant
so is $\Sigma(X)$. Therefore we may later (tacitly) replace $X$ and
$Y$ by suitable open dense subsets.

\begin{lemma}\label{lemma:Gstrich}

  Let $f:X\to Y$ be dominant or let $f: Y\to X$ be injective. Assume
  $G'_Y\subseteq G'_X$. Then there is exists a distinguished
  homomorphism $\eta:\Gv_Y\to\Gv_X$.

\end{lemma}

\begin{proof}

  We claim that $\Xi(Y)\subseteq\Xi(X)$ in both cases. This is clear
  if $f$ is dominant since the pull-back of a $B$-semiinvariant is
  again a $B$-semiinvariant for the same character. For $f$ injective
  let $p:\Xq\to X$ be the normalization and let $\Yq\subseteq\Xq$ be a
  component of $p^{-1}(Y)$ mapping dominantly to $Y$. By
  \cite{KnopLV}*{Thm.~1.3 \textit{b)}}, every $B$-semiinvariant
  rational function on $\Yq$ extends to a $B$-semiinvariant rational
  function on $\Xq$. Since the character remains unchanged we get
  $\Xi(Y)\subseteq\Xi(\Yq)\subseteq\Xi(\Xq)=\Xi(X)$.

  It is a general fact that if $H\subseteq G$ is reductive then the
  coroot lattice of $H$ is contained in the coroot lattice of $G$
  (look at simply connected covers). Applying this to
  $G'_Y\subseteq G'_X$ we get $\ZZ\Sigma(Y)\subseteq\ZZ\Sigma(X)$ and
  therefore
  \[\label{eq:tildeXi}
    \XiS(Y)\subseteq\XiS(X)
  \]
  This inclusion induces a homomorphism of maximal tori
  $\Av_Y\to\Av_X$. Because $G^*_X$ is generated by $G'_X$ and
  $\phi_X(\Av_X)$ (and similarly for $Y$) it follows that
  $G^*_Y\subseteq G^*_X$.

  Finally, the coweight lattice of
  $G^{*\vee}_Y:=\phi_X^{-1}(G^*_Y)^0\subseteq\Gv_X$ is
  $\XiS(Y)_\QQ\cap\XiS(X)$. By \eqref{eq:tildeXi}, it contains the
  coweight lattice $\XiS(Y)$ of $\Gv_Y$. Hence the inclusion
  $G^*_Y\into G^*_X$ lifts to an isogeny $\Gv_Y\to G^{*\vee}_Y$
  yielding the desired homomorphism $\eta:\Gv_Y\to\Gv_X$.
\end{proof}

The following comparison result will be crucial later on. It is a more
precise version of \cref{thm:main3} in case $Y$ is of codimension $1$.

\begin{theorem}\label{thm:comparison}

  Let $X$ be a normal $G$-variety and let $Y\subset X$ be a
  $G$-invariant irreducible subvariety of codimension $1$. Then
  $\Sigma(Y)\subseteq\Sigma(X)$ and therefore $G'_Y\subseteq
  G'_X$. Moreover, if the valuation $v:=v_Y$ induced by $Y$ is
  non-central then $\cN(Y)=\cN(X)$. Otherwise, $\cN(Y)=\cN(X)/\QQ v$
  and
  \[\label{eq:sigmaX0}
    \Sigma(Y)=\{\sigma\in\Sigma(X)\mid v(\sigma)=0\}.
  \]

\end{theorem}

\begin{proof}

  This is essentially proved in \cite{KnopIB}. Assume first that $v$
  is central, i.e., that the restriction of $v$ to $k(X)^B$ is trivial
  (that's automatic if $X$ is spherical). Then there is a surjective
  homomorphism
  \[
    \cN(X)\auf\cN(Y)
  \]
  with kernel $\QQ v$ such that $\cZ(Y)$ is the image of $\cZ(X)$
  (loc.cit.~Satz 7.5.2 with $v_0=o$). Thus, the preimage of $\cZ(Y)$
  is the cone $\cZ(X)+\QQ v$. Because of $v\in\cZ(X)$, this cone is
  defined by the inequalities $\sigma\le0$ with $\sigma\in\Sigma(X)$
  and $v(\sigma)=0$. This proves \eqref{eq:sigmaX0}.

  Assume now that $v$ is not central and let $v_0$ be the restriction
  of $v$ to $k(X)^B$. Let $\cZ_{v_0}$ be the set of $G$-invariant
  valuations whose restriction of $k(X)^B$ is a multiple of
  $v_0$. Then $\cZ_{v_0}$ can be identified with a convex cone in some
  $\QQ$-vector space $\cN_{v_0}$. Moreover, $\cN(X)$ is a hyperplane
  of $\cN_{v_0}$ such that $\cZ_{v_0}\cap\cN(X)=\cZ(X)$ (see the exact
  sequence in loc.cit.~\S5 where $\cN_{v_0}$ is corresponds to
  $\Hom(\cQ_{v_0}(K),\QQ)$).
  
  There is a surjective homomorphism (loc.cit.~Satz~7.5.2)
  \[
    \cN_{v_0}\auf\cN(Y)
  \]
  with kernel $\QQ v$ such that $\cZ(Y)$ is the image of
  $\cZ_{v_0}$. Since by assumption $v\not\in\cN(X)$ we have
  $\cN(X)\overset\sim\to\cN(Y)$, as asserted.

  It is a non-trivial fact (loc.cit.~Satz~9.2.2) that as a cone
  $\cZ_{v_0}$ is generated by $\cZ(X)$ along with one extremal
  non-central valuation $v_e$, i.e.,
  \[
    \cZ_{v_0}=\cZ(X)+\QQ_{\ge0}v_e.
  \]
  Let $v=v_1+cv_e$ with $v_1\in\cZ(X)$ and $c>0$. Then the preimage of
  $\cZ(Y)$ in $\cN_{v_0}$ equals
  \[
    \cZ_{v_0}+\QQ v=\cZ(X)+\QQ_{\ge0}v_e+\QQ v=\cZ(X)+\QQ v_1+\QQ v_e.
  \]
  This shows that
  \[
    \cZ(Y)=(\cZ_{v_0}+\QQ v)\cap\cN(X)=\cZ(X)+\QQ v_1
  \]
  is defined by the inequalities $\sigma\le0$ with
  $\sigma\in\Sigma(X)$ and $v_1(\sigma)=0$. In particular
  $\Sigma(Y)\subseteq\Sigma(X)$.
\end{proof}

At this point we already have a

\begin{proof}[Proof of \cref{thm:main3} assuming \cref{thm:main2}] We
  may assume that $Y$ is a subvariety of $X$. It suffices to construct
  a normal $G$-variety $\Xq$, a birational $G$-morphism
  $\pi:\Xq\to X$, and a $G$-stable subvariety $\Yq\subset\Xq$ of
  codimension $1$ which maps dominantly to $Y$. In fact, in this case
  we have $G'_Y\subseteq G'_{\Yq}\subseteq G'_{\Xq}=G'_X$ by
  \cref{thm:main2} and \cref{thm:comparison}. Then
  \cref{lemma:Gstrich} yields a distinguished homomorphism
  $\Gv_Y\to \Gv_X$.

  To construct $\Xq$ let $p:X_1\to X$ be the normalization of $X$ and
  let $Y_1\subseteq X_1$ be a component of $p^{-1}(Y)$ which maps
  surjectively to $Y$. Next, let $X_2\to X_1$ be the blow up of $X_1$
  in $Y_1$ and let $Y_2\subset X_2$ be a component of the exceptional
  divisor. Finally, the normalization $p_2:\Xq\to X_2$ with
  $\Yq\subset\Xq$ a component of $p_2^{-1}(Y_2)$ meets all
  requirements.
\end{proof}

For the next step, recall that a homogeneous variety $G/H$ is
\emph{parabolically induced} if there is a proper parabolic subgroup
$Q\subset G$ with $Q_u\subseteq H\subseteq Q$. It is \emph{cuspidal}
if is not parabolically induced and if $H$ does not contain a simple
factor of $G$.

\begin{lemma}\label{lemma:reduction}

  Assume $G'_Y\subseteq G'_X$ in the following situation:

  \begin{itemize}

  \item $G$ is of adjoint type,

  \item $Y=G/H$ is homogeneous, spherical and cuspidal of rank $1$,
    and $H$ is connected.

  \item $X=G/P$ where $P\subset H$ is a maximal parabolic subgroup.

  \end{itemize}

  Then $G_Y'\subseteq G_X'$ for all $G$-varieties $X$, $Y$ and all
  dominant $G$-morphisms $X\to Y$.

\end{lemma}

\begin{proof}

  We will prove the assertion by induction on $\dim X+\dim G$. For
  this let $f:X\to Y$ be an arbitrary dominant $G$-morphism.

  \emph{Reduction to $\rk Y=\#\Sigma(Y)=1$:} Assume $\rk Y\ge2$. Every
  $\tau\in\Sigma(Y)$ is a simple coroot of $G'_Y$ and therefore
  induces a semisimple rank-$1$-subgroup $G'_Y(\tau)\subseteq
  G'_Y$. Since the subgroups of this form generate $G'_Y$ it suffices
  to prove $G'_Y(\tau)\subseteq G'_X$ for all $\tau$.

  If $\Sigma(Y)=\leer$ then $G'_Y=1$ and there is nothing to prove. So
  fix $\tau\in\Sigma(Y)$. Then $\tau$ defines a codimension-$1$-face
  $\cF$ of the valuation cone $\cZ(Y)$. Since $\dim\cF=\rk Y-1\ge1$
  there is a non-trivial valuation $v$ in the relative interior of
  $\cF$. Let $Y\into\Yq=Y\cup Y_0$ be the smooth equivariant embedding
  where $Y_0$ is an irreducible divisor such that $v_{Y_0}$ is a
  rational multiple of $v$. Then $\rk Y_0=\rk Y-1$ and
  $\Sigma(Y_0)=\{\tau\}$ by \cref{thm:comparison}. By
  \cite{KnopIB}*{Kor.~3.2} there exists a lift of $v$ to a (possibly
  non-central) equivariant valuation $\overline v$ of $X$. This gives
  rise to a similar embedding $X\into\Xq=X\cup X_0$ such that $f$
  extends to a morphism $\Xq\to\Yq$ which maps $X_0$ dominantly to
  $Y_0$. \cref{thm:comparison} implies that
  $\Sigma(X_0)\subseteq\Sigma(X)$. Hence we have
  \[
    G'_Y(\tau)=G'_{Y_0}\text{ and }G'_{X_0}\subseteq G'_X.
  \]
  By induction we have $G'_{Y_0}\subseteq G'_{X_0}$ which proves the
  assertion.

  \emph{Reduction to $G$ semisimple:} Let $Z=Z(G)^0$ be the connected
  center of $G$. If $Z$ acts trivially on $X$ then one can replace $G$
  by the semisimple group $G/Z$. Otherwise, consider the morphism
  $X_0:=X'/Z\to Y_0:=Y'/Z$ where $X'\subseteq X$ and $Y'\subseteq Y$
  are non-empty, open, and $G$-stable such that the $Z$-orbit spaces
  exist (these exist by \cite{Rosenlicht}*{Thm.~2}). Because of
  $\Sigma(X_0)=\Sigma(X)$ and $\Sigma(Y_0)=\Sigma(Y)$ by
  \cite{KnopIB}*{Satz~8.1.4} we have $G'_Y\subseteq G'_X$ if and only
  if $G'_{Y_0}\subseteq G'_{X_0}$. The latter holds by induction.

  \emph{Reduction to $X$ and $Y$ homogeneous:} Let $Y_0\subseteq Y$ be
  a general orbit. Then $\Sigma(Y_0)=\Sigma(Y)$ by
  \cite{KnopWuM}*{Satz~6.5.4}. Let $X_0\subseteq X$ be a general orbit
  in the preimage of $Y_0$ in $X$. Then $X_0$ is also a general orbit
  of $X$ and therefore $\Sigma(X_0)=\Sigma(X)$. This proves the
  assertion by induction unless $X=X_0$ and $Y=Y_0$.

  \emph{Reduction to $f$ proper:} We may assume that $X$ and $Y$ are
  homogeneous. If $f$ is not proper choose a normal equivariant
  embedding $X\into\Xq$ such that $f$ extends to a proper morphism
  $\Xq\to Y$.  Let $X_0$ be a component of $\Xq\setminus X$. By
  blowing up $X$ in $X_0$ and normalizing, if necessary, we may assume
  that $X_0$ is a $G$-invariant irreducible divisor. Then
  $\Sigma(X_0)\subseteq\Sigma(X)$ by \cref{thm:comparison} and
  therefore $G'_{X_0}\subseteq G'_X$. The assertion follows by
  applying the induction hypotheses to $X_0\to Y$.

  Because of the last steps we may assume that $X=G/P$, $Y=G/H$ with
  $P^0\subseteq H^0$ parabolic and $\rk Y=1$.

  \emph{Reduction to $P$ and $H$ connected:} Follows from the fact
  that $W(X)$, hence $\Sigma(X)$, hence $G'_X$ is invariant under
  étale maps (see \cite{KnopWuM}*{Satz~6.5.3}).

  \emph{Reduction to $P\subset H$ maximal parabolic:} Assume that
  there is a parabolic $Q$ with $P\subset Q\subset H$ and put
  $Z:=G/Q$. We may assume $P$ to be maximal parabolic in $Q$. By
  induction on the morphism $Z\to Y$ it suffices to prove
  $G'_Z\subseteq G'_X$ for the morphism $X\to Z$. This is indeed
  implied by the first reduction step unless $\rk Z=1$.

  \emph{Reduction to $H$ cuspidal:} Suppose there is a parabolic
  subgroup $Q=LQ_u\subset G$ with $Q_u\subseteq H\subseteq Q$. Then
  $Q_u\subseteq H_u$ and $H_u\subseteq P_u$ (since $P$ is parabolic in
  $H$). This shows that $P$ is also induced by $Q$. The
  $L=Q/Q_u$-varieties $X_0=Q/P=L/(P\cap L)$ and $Y_0=Q/H=L/(H\cap L)$
  have $\Sigma(X_0)=\Sigma(X)$ and $\Sigma(Y_0)=\Sigma(X)$ (see, e.g.,
  \cite{KK}~Prop.~8.2). Then we conclude by induction. If $H$ contains
  a simple factor $G_1$ of $G$ then there are decompositions
  $G=G_1\cdot G_2$ and $H=G_1\cdot H_2$. A maximal parabolic subgroup
  of $H$ is either of the form $P_1\cdot H_1$ (in which case
  $\Sigma(X)=\Sigma(Y)$) or $G_1\cdot P_2$ (in which case $G_1$ acts
  trivially on both $X$ and $Y$ and we may replace $G$ by $G/G_1$).

  \emph{Reduction to $H$ spherical:} The only cuspidal homogeneous
  rank-$1$-varieties which are not spherical are of the form $G/H$
  where $G=\SL(2)$ and $H$ is finite (\cite{PanyushevRankOne}). By
  previous reduction steps we may assume that $H$ is connected (hence
  trivial) and contains a proper parabolic subgroup. So this case does
  not occur.

  This finishes the reduction of a general dominant morphism to the
  situation in the Lemma.
\end{proof}

\section{The rank-$1$-case}
\label{sec:rankone}

Using \cref{lemma:reduction}, the proof of \cref{thm:main2} is now
reduced to the cases where $G$ is of adjoint type, $Y=G/H$ is
homogeneous, spherical and cuspidal of rank $1$, with $H$ connected,
and $X=G/P$ where $P\subset H$ is a maximal parabolic subgroup.

The classification of all possible pairs $(G,H)$ is due to Akhiezer
\cite{Akhiezer} (see also Brion's simplification \cite{BrionRank1})
and is reproduced in \cref{tab:rankone} below. In the case $\sB'_n$,
the group $P_n$ denotes a maximal parabolic subgroup of $\SO(2n)$
whose Levi part is $GL(n)$. In $\sC'_n$, the group
$B_2\subseteq\Sp(2)$ is a Borel subgroup. Finally $U_3$ in case
$\sG'_2$ is a $3$-dimensional unipotent group. The two columns on the
right will be used in the final step of the proof of \cref{thm:main2}.

\begin{table}
  \caption{}
  \label{tab:rankone}
  \begin{tabular}{H|HH|HHH}

    &G&H&&\tau\ass&\Sa\\

    \hline\hline

    \sA_1&\PGL(2)&\G_m&\\

    \hline

    \sA_{n\ge2}&\PGL(n+1)&\G_m\SL(n)&(1)&\siv_1{+}\siv_2&\bild{\Alinks\I1\Arechts\I2}\\

    \hline

    \sB_{n\ge2}&\SO(2n+1)&\SO(2n)&(1)&2\siv_1{+}\siv_2&\bild{\Alinks\I1\Crechts\I2}\\
    &&&(2)&\siv_1&\bild{\Alinks\I1}\\

    \hline

    \sB'_{n\ge2}&\SO(2n+1)&P_n&(1)&2\siv_1{+}\siv_2&\bild{\Alinks\I1\Crechts\I2}\\
    &&&(2)&2\siv_1{+}\siv_2&\bild{\Alinks\I1\Crechts\I2}\\

    \hline

    \sC_{n\ge3}&\PSp(2n)&\Sp(2)\Sp(2n-2)&(1)&\siv_1&\bild{\Alinks\I1}\\
    &&&(2)&\siv_1{+}\siv_2&\bild{\Alinks\I1\Arechts\I2}\\
    &&&(3)&\gv_1{+}2\siv_2{+}\gv_2{+}\siv_3&
                                             \bild{\Alinks\I4\Arechts\I2\Drechts\def\Punkt{(1.5,0.5)}\IndexRechts{\gv_1}\def\Punkt{(1.5,-0.5)}\IndexRechts{\gv_2}}\\
    &&&(4)&\gv_1{+}2\siv_2{+}2\gv_2&\bild{\Alinks\Ig1\Arechts\I2\Brechts\Ig2}\\

    \hline

    \sC'_{n\ge3}&\PSp(2n)&B_2\Sp(2n-2)&(1)&\siv_1{+}\siv_2{+}\siv_3&\bild{\Alinks\I1\Arechts\I2\Arechts\I3}\\
    &&&(2)&\siv_1{+}2\siv_2{+}\siv_3{+}\siv_4&
                                               \bild{\Alinks\I4\Arechts\I2\Drechts\def\Punkt{(1.5,0.5)}\IndexRechts{\siv_1}\def\Punkt{(1.5,-0.5)}\IndexRechts{\siv_3}}\\
    &&&(3)&\siv_1{+}2\siv_2{+}2\siv_3&\bild{\Alinks\I1\Arechts\I2\Brechts\I3}\\

    \hline

    \sF_4&\sF_4&\Spin(9)&(1)&\gv_2{+}2\siv_2{+}2\gv_1&\bild{\Alinks\Ig2\Arechts\I2\Brechts\Ig1}\\
    &&&(2)&\siv_2{+}2\siv_3{+}2\siv_1&\bild{\Alinks\I2\Arechts\I1\Brechts\I3}\\
    &&&(3)&\siv_3{+}2\siv_4{+}2\siv_2{+}2\siv_1&\bild{\Alinks\I3\Arechts\I4\Arechts\I2\Brechts\I1}\\
    &&&(4)&\siv_2{+}\siv_1{+}\siv_3&\bild{\Alinks\I1\Arechts\I2\Arechts\I3}\\

    \hline

    \sG_2&\sG_2&\SL(3)&(1)&\siv_1+\siv_2&\bild{\Alinks\I1\Arechts\I2}\\

    \hline

    \sG'_2&\sG_2&\G_m\SL(2)U_3&(1)&\siv_1{+}3\siv_2&\bild{\Alinks\I1\def\Punkt{(1,0)}\Glinks\I2}\\

    \hline\hline

    \sD_{n\ge2}&\PSO(2n)&\SO(2n-1)&(1)&\{\gv_1{+}\siv_1,\siv_1{+}\gv_2\}&\bild{\Alinks\Ig1\Arechts\I1\Arechts\Ig2}\\
    &&&(2)&\{\siv_1,\siv_2\}&\bild{\Alinks\I1\def\Punkt{(1,0)}\Alinks\I2}\\

    \hline

    \sB''_3&\SO(7)&\sG_2&(1)&\{\siv_1{+}\siv_3,\siv_2\}&\bild{\Alinks\I1\Crechts\I3\def\Punkt{(2,0)}\Alinks\I2}\\
    &&&(2)&\{\siv_1{+}\siv_2{+}\siv_3,2\siv_2{+}\siv_3\}&\bild{\Alinks\I1\Arechts\I2\Crechts\I3}\\

    \hline

  \end{tabular}
\end{table}

We have $\Sigma(G/H)=\{\tau\}$ and we need to compute
$\Sigma=\Sigma(G/P)$ for all maximal parabolic subgroups $P\subset
H$. This is done in \cref{sec:appendix}. All varieties $G/P$ turn out
to be spherical, even wonderful, a fact for which we don't have a
conceptual argument.

For every spherical root $\sigma$ define its set $\sigma\ass$ of
associated roots as
\[\label{eq:ass}
  \sigma\ass=
  \begin{cases}
    \{\siv\}&\text{if }\sigma\in\Phi,\\
    \{\gv_1,\gv_2\}&\text{otherwise (with $\gv_i$ as in
      \cref{tab:assocroots}).}
  \end{cases}
\]
Put $\Sa:=\cup_{\sigma\in\Sigma}\sigma\ass$. It was shown in
\cite{KnopSchalke} that $\Sa$ is the basis of a maximal rank subgroup
$\Ga_X\subseteq\Gv$. Moreover, the root system of $\Gv_X$ is obtained
from that of $\Ga_X$ by a process called ``folding''. Let $\Phi\ass_X$
be the set of roots of $\Ga_X$.

From \cref{tab:parabolic} one can read off $\Sa$ and $\tau\ass$ as a
linear combination of $\Sa$. The result is recorded in the two right
hand columns of \cref{tab:rankone}. As an example, consider case
$\sC_n(4)$. Here $\sigma_1=\gamma_1+\gamma_2$ with $\gamma_1=\alpha_1$
and $\gamma_2=\alpha_n$. Since $\sigma_2$ is a root we have
$\Sa=\{\gv_1,\siv_2,\gv_2\}$ which is a basis of a root system of type
$\sB_3$. Moreover, one verifies
$\tau\ass=\av_1+2\av_2+\ldots+2\av_{n-1}+2\av_n=\gv_1+2\siv_2+2\gv_2$.

Now it is easy to finish the proof of \cref{thm:main2}.

First, we consider the case $\Sa=\Siv$ (recognizable by the
non-appearance of $\gv_i$'s). Here one checks that
$\tau\ass\subseteq\Phi\ass_X$ which implies
$\Gv_Y\subseteq\Ga_X=\Gv_X$.

Next assume that $\Sa\ne\Siv$ but $\tau\ass=\{\tv\}$.  Here, one
checks that $\tv$ is actually the highest root of $\Phi\ass_X$. Since
all simple roots of $\Ga_X$ restrict to simple roots of $\Gv_X$, there
is no other root of $\Ga_X$ which has the same restriction as
$\tau\ass$. This implies $\fgv_{X,\tau}=\fg\ass_{X,\tau}=\fgv_\tau$
and therefore $\Gv_Y\subseteq\Gv_X$.

The only case remaining is that of $\sD_n(1)$ depending on a parameter
$\nu\in\{1,\ldots,n-2\}$. It suffices to prove
\[\label{eq:Dn}
  \fgv_{\tv}=[\fgv_{\siv_1},\fgv_{\siv_2}]
\]
since then $\fgv_{\tv}\subseteq\fgv_X$ and therefore
$\Gv_Y\subseteq\Gv_X$.

Using the standard basis $\e_i$ for the weight lattice of $\sD_n$ and
the normalization \eqref{eq:defsv} we have
\[
  \fgv_{\tv}=[\fgv_{\e_1-\e_{n-1}},E]\text{ with
  }E:=e_{\e_{n-1}-\e_n}-e_{\e_{n-1}+\e_n}.
\]
If $\nu=n-2$ then $\fgv_{\siv_2}=\CC E$ and
$\siv_1=\e_1-\e_{\nu+1}=\e_1-\e_{n-1}$ which proves
\eqref{eq:Dn}. Otherwise, we have
\[
  \fgv_{\siv_2}=[\fgv_{\e_{\nu+1}-\e_{n-1}},E]
\]
and therefore
\[
  [\fgv_{\siv_1},\fgv_{\siv_2}]=[\fgv_{\e_1-\e_{\nu+1}},[\fgv_{\e_{\nu+1}-\e_{n-1}},E]]=
  [\fgv_{\e_1-\e_{n-1}},E]=\fgv_{\tv}.
\]
\cref{thm:main2} is proved.\qed

\section{Appendix: Maximal parabolics in
  rank-$1$-subgroups}\label{sec:appendix}

In the following, we use the classification of spherical varieties
using Luna diagrams due to Luna \cite{Luna}, Losev \cite{Losev}, and
Bravi-Pezzini \cite{BraviPezzini}. A very good introduction to this
topic can be found in \cite{BraviLuna}.

\cref{tab:parabolic} below lists the Luna diagrams of all cuspidal
rank-$1$-varieties $Y=G/H$ ($G$ adjoint, $H$ connected). For each such
diagram we list a number of further Luna diagrams. We claim that these
classify all varieties $X=G/P$ with $P\subset H$ maximal
parabolic.

Along with the diagram of $X$ we are also giving the complete generalized
Cartan matrix so that the ``decorations'' of the diagrams by arrow
heads ``$<$'' or ``$>$'' are not needed. The rows of the Cartan matrix are
labelled by the spherical roots $\sigma_i\in\Sigma:=\Sigma(X)$. The
columns correspond to the colors, i.e., to the $B$-invariant
irreducible divisors $D_j$ of $X$. They also correspond to the circles
(filled or empty) in the Luna diagram. The index $j$ of $D_j$ means
that $D_j$ is attached to the simple root $\a_j$. The entries of
the Cartan matrix are the numbers $v_{D_j}(f_{\sigma_i})\in\ZZ$ where
$f_{\sigma_i}\in k(X)$ is a $B$-semiinvariant for the character
$\sigma_i$.

The claim can be verified in several easy steps:

1. First, one checks that all diagrams and Cartan matrices satisfy
Luna's axioms. Thus, each belongs to a unique spherical (even
wonderful) variety $X=G/P$.

2. Let $\cD_0$ be the set of colors which are printed in boldface. The
corresponding columns sum up to $0$ which shows that $\cD_0$ is
distinguished in the sense of \cite{BraviLuna}*{2.3}. Therefore,
$\cD_0$ defines a $G$-morphism $X\to Y'=G/H'$ with
$P\subseteq H'\subseteq G$ and $H'/P$ is connected.

3. Next one uses \cite{BraviLuna}*{2.3} to verify that the spherical
systems of $Y$ and $Y'$ coincide which then implies that $H'$ is
conjugate to $H$. To do this one shows that $\tau$ (whose coordinates
in terms of the $\sigma_i$ are provided in the leftmost column)
generates the orthogonal complement of the boldface columns. One also
has to observe that the colors not in $\cD_0$ correspond to the colors
of $Y$.

4. That $P$ is parabolic in $H$ is equivalent to $G/P\to G/H$ being
proper which is equivalent to no $G$-invariant valuation of $G/P$
restricting to the trivial valuation of $G/H$. This in turn translates
into $\tau$ being a linear combination of the $\sigma_i$ with strictly
positive coefficients. This is clear from looking at the leftmost
column.

5. The submatrix given by the boldface entries is always a square
matrix of defect $1$. Hence the columns of every proper subset of
$\cD_0$ are linear independent which shows that such a subset in not
distinguished. This means that $P$ is maximal proper subgroup of $H$.

6. The preceding steps show that $P$ is a maximal parabolic in $H$. To
see that all of them are listed one checks that the number of items in
the table equals the number of $G$-conjugacy classes of maximal
parabolics of $H$. To do this one can consult \cref{tab:rankone} for
$H$. In most cases this number equals the number of maximal parabolics
of $H$. Only in the cases $\sB_n$ and $\sG_2$ there is an element of
$N_G(H)$ acting as an outer automorphism on $H$. This results in two
non-conjugate maximal parabolics of $H$ being conjugate in $G$
resulting in one item less.

\def\scalar{0.6}

\begin{longtable}{HH|H}

\caption*{Table \ref{tab:parabolic}.}\label{tab:parabolic}

\\\hline

\sA_n&
\begin{array}[t]{l}
\bild{\Alinks\WurzelAn}\\
\tau=\a_1\+\a_n\\
\end{array}\\

(1)&
\begin{array}[t]{ll}
\bild{\Alinks\WurzelAn\IndexUnten{\a_\nu\rlap{\ (\nu=1,\ldots,n-1)}}\Arechts\WurzelAn}\\
\sigma_1=\a_1\+\a_\nu\\
\sigma_2=\a_{\nu+1}\+\a_n\\
\end{array}

&\begin{array}[t]{l|l|r|R|R|r|}
&&D_1&D_\nu&D_{\nu+1}&D_n\\
\hline
1&\sigma_1&1&1&-1&0\\
1&\sigma_2&0&-1&1&1\\
\end{array}\\

\hline

\sB_n&
\begin{array}[t]{l}
\bild{\FarbeD\Arechts\ddd\Brechts}\\
\tau=\a_1\+\a_n\\
\end{array}\\

(1)&
\begin{array}[t]{l}
\bild{\Alinks\WurzelAn\IndexUnten{\a_\nu\rlap{\ (\nu=1,\ldots,n-2)}}\Arechts\FarbeD\ddd\Brechts}\\
\sigma_1=\a_1\+\a_\nu\\
\sigma_2=\a_{\nu+1}\+\a_n\\
\end{array}&

\begin{array}[t]{l|l|r|R|R|}
&&D_1&D_\nu&D_{\nu+1}\\
\hline
1&\sigma_1&1&1&-1\\
1&\sigma_2&0&-1&1\\
\end{array}\\

\noalign{\medskip}

(2)&
\begin{array}[t]{l}
\bild{\FarbeD\Arechts\ddd\Brechts\FarbeM}\\
\sigma_1=\a_1\+\a_n\\
\end{array}&

\begin{array}[t]{l|l|r|R|}
&&D_1&D_n\\
\hline
1&\sigma_1&1&0\\
\end{array}\\

\hline

\sB'_n&
\begin{array}[t]{l}
\bild{\FarbeD\Arechts\ddd\Brechts\FarbeM}\\
\tau=\a_1\+\a_n\\
\end{array}\\

(1)&
\begin{array}[t]{l}
\bild{\Alinks\WurzelAn\IndexUnten{\a_\nu\rlap{\ (\nu=1,\ldots,n-2)}}\Arechts\FarbeD\ddd\Brechts\FarbeM}\\
\sigma_1=\a_1\+\a_\nu\\
\sigma_2=\a_{\nu+1}\+\a_n\\
\end{array}&

\begin{array}[t]{l|l|r|R|R|r|}
&&D_1&D_\nu&D_{\nu+1}&D_n\\
\hline
1&\sigma_1&1&1&-1&0\\
1&\sigma_2&0&-1&1&0\\
\end{array}\\

(2)&
\begin{array}[t]{l}
\bild{\Alinks\WurzelAn\Brechts\FarbeU\FarbeO}\\
\sigma_1=\a_1\+\a_{n-1}\\
\sigma_2=\a_n\\
\end{array}&

\begin{array}[t]{l|l|r|R|R|r|}
&&D_1&D_{n-1}&D_n^+&D_n^-\\
\hline
1&\sigma_1&1&1&-1&-1\\
1&\sigma_2&0&-1&1&1\\
\end{array}\\

\hline

\sC_n&
\begin{array}[t]{l}
\bild{\Alinks\Arechts\FarbeD\Arechts\ddd\Crechts}\\
\tau=\a_1+2\a_2\+2\a_{n-1}+\a_n\\
\end{array}\\

(1)&
\begin{array}[t]{l}
\bild{\Alinks\FarbeM\Arechts\FarbeD\Arechts\ddd\Crechts}\\
\sigma_1=\a_1+2\a_2\+2\a_{n-1}+\a_n\\
\end{array}&

\begin{array}[t]{l|l|R|r|}
&&D_1&D_2\\
\hline
1&\sigma_1&0&1\\
\end{array}\\

(2)&
\begin{array}[t]{l}
\bild{\Alinks\FarbeM\zz\Arechts\FarbeM\Arechts\FarbeD\Arechts\ddd\Crechts}\\
\sigma_1=\a_1+\a_2\\
\sigma_2=\a_2+2\a_3\+2\a_{n-1}+\a_n\\
\end{array}&

\begin{array}[t]{l|l|R|r|R|}
&&D_1&D_2&D_3\\
\hline
1&\sigma_1&1&1&-1\\
1&\sigma_2&-1&0&1\\
\end{array}\\

(3)&
\begin{array}[t]{l}
\bild{\Alinks\FarbeM\Arechts\WurzelAn\IndexUnten{\a_\nu\rlap{\ \ \ \ \ (\nu=2,\ldots,n-2)}}\Arechts\FarbeM\Arechts\FarbeD\Arechts\ddd\Crechts\draw (0,-0.2)--++(0,-0.2)--++(6,0)--++(0,0.2);}\\
\sigma_1=\a_1+\a_{\nu+1}\\
\sigma_2=\a_2\+\a_\nu\\
\sigma_3=\a_{\nu+1}+2\a_{\nu+2}\+2\a_{n-1}+\a_n\\
\end{array}&

\begin{array}[t]{l|l|R|r|R|R|}
&&D_1&D_2&D_\nu&D_{\nu+2}\\
\hline
1&\sigma_1&2&-1&-1&-1\\
2&\sigma_2&-1&1&1&0\\
1&\sigma_3&0&0&-1&1\\
\end{array}\\

(4)&
\begin{array}[t]{l}
\bild{\Alinks\FarbeM\Arechts\WurzelAn\Crechts\FarbeM\draw (0,-0.2)--++(0,-0.2)--++(6,0)--++(0,0.2);}\\
\sigma_1=\a_1+\a_n\\
\sigma_2=\a_2\+\a_{n-1}\\
\end{array}&

\begin{array}[t]{l|l|R|r|R|}
&&D_1&D_2&D_{n-1}\\
\hline
1&\sigma_1&2&-1&-2\\
2&\sigma_2&-1&1&1\\
\end{array}\\

\hline

\sC'_n&
\begin{array}[t]{l}
\bild{\Alinks\FarbeM\Arechts\FarbeD\Arechts\ddd\Crechts}\\
\tau=\a_1+2\a_2\+2\a_{n-1}+\a_n\\
\end{array}\\

(1)&
\begin{array}[t]{l}
\bild{\Alinks\FarbeU\FarbeO\Arechts\FarbeU\FarbeO\Arechts\FarbeD\Arechts\ddd\Crechts}\\
\sigma_1=\a_1\\
\sigma_2=\a_2\\
\sigma_3=\a_2+2\a_3\+2\a_{n-1}+\a_n\\
\end{array}&

\begin{array}[t]{l|l|R|r|r|R|R|}
&&D_1^+&D_1^-&D_2^+&D_2^-&D_3\\
\hline
1&\sigma_1&1&1&0&-1&0\\
1&\sigma_2&0&-1&1&1&-1\\
1&\sigma_3&-1&0&0&0&1\\
\end{array}\\

(2)&\begin{array}[t]{l}
\bild{\Alinks\FarbeU\FarbeO\Arechts\WurzelAn\IndexUnten{\a_\nu\rlap{\ \ \ \ \ (\nu=2,\ldots,n-2)}}\Arechts\FarbeU\FarbeO\Arechts\FarbeD\Arechts\ddd\Crechts\draw (0,-0.57)--++(0,-0.2)--++(6,0)--++(0,0.2);}\\
\sigma_1=\a_1\\
\sigma_2=\a_2\+\a_\nu\\
\sigma_3=\a_{\nu+1}\\
\sigma_4=\a_{\nu+1}+2\a_{\nu+2}\+2\a_{n-1}+\a_n\\
\end{array}&

\begin{array}[t]{l|l|r|R|r|R|R|R|}
&&D_1^+&D_1^-&D_2&D_\nu&D_{\nu+1}^+&D_{\nu+2}\\
\hline
1&\sigma_1&1&1&-1&0&-1&0\\
2&\sigma_2&0&-1&1&1&0&0\\
1&\sigma_3&-1&1&0&-1&1&-1\\
1&\sigma_4&0&0&0&-1&0&1\\
\end{array}\\

(3)&
\begin{array}[t]{l}
\bild{\Alinks\FarbeU\FarbeO\Arechts\WurzelAn\Crechts\FarbeU\FarbeO\draw (0,-0.57)--++(0,-0.2)--++(6,0)--++(0,0.2);}\\
\sigma_1=\a_1\\
\sigma_2=\a_2\+\a_{n-1}\\
\sigma_3=\a_n\\
\end{array}&

\begin{array}[t]{l|l|r|R|r|R|R|}
&&D_1^+&D_1^-&D_2&D_{n-1}&D_n^+\\
\hline
1&\sigma_1&1&1&-1&0&-1\\
2&\sigma_2&0&-1&1&1&0\\
1&\sigma_3&-1&1&0&-2&1\\
\end{array}\\

\hline

\sF_4&
\begin{array}[t]{l}
\bild{\Alinks\Arechts\Brechts\Arechts\FarbeD}\\
\tau=\a_1+2\a_2+3\a_3+2\a_4\\
\end{array}\\

(1)&
\begin{array}[t]{l}
\bild{\Alinks\Arechts\Brechts\FarbeD\Arechts\FarbeU\FarbeO}\\
\sigma_1=\a_1+2\a_2+3\a_3\\
\sigma_2=\a_4\\
\end{array}&

\begin{array}[t]{l|l|R|R|r|}
&&D_3&D_4^+&D_4^-\\
\hline
1&\sigma_1&2&-2&-1\\
2&\sigma_2&-1&1&1\\
\end{array}\\

(2)&
\begin{array}[t]{l}
\bild{\Alinks\FarbeM\zz\Arechts\FarbeD\Brechts\FarbeM\zz\Arechts\FarbeM}\\
\sigma_1=\a_1+\a_2\\
\sigma_2=\a_2+\a_3\\
\sigma_3=\a_3+\a_4\\
\end{array}&

\begin{array}[t]{l|l|R|R|R|r|}
&&D_1&D_2&D_3&D_4\\
\hline
1&\sigma_1&1&1&-2&0\\
1&\sigma_2&-1&1&0&-1\\
2&\sigma_3&0&-1&1&1\\
\end{array}\\

(3)&
\begin{array}[t]{l}
\bild{\Alinks\FarbeU\FarbeO\Arechts\FarbeD\Brechts\FarbeU\FarbeO\Arechts\FarbeU\FarbeO\draw (0,0.57)--++(0,0.2)--++(2,0)--++(0,-0.2);\draw (0,-0.57)--++(0,-0.2)--++(3,0)--++(0,0.2);}\\
\sigma_1=\a_1\\
\sigma_2=\a_2+\a_3\\
\sigma_3=\a_3\\
\sigma_4=\a_4\\
\end{array}&

\begin{array}[t]{l|l|R|R|R|R|r|}
&&D_1^+&D_1^-&D_2&D_3^-&D_4^+\\
\hline
1&\sigma_1&1&1&-1&-1&-1\\
2&\sigma_2&0&-1&1&0&0\\
1&\sigma_3&1&-1&-1&1&0\\
2&\sigma_4&-1&1&0&0&1\\
\end{array}\\

(4)&
\begin{array}[t]{l}
\bild{\FarbeD\Arechts\Brechts\FarbeD\Arechts\FarbeU\FarbeO}\\
\sigma_1=\a_1+\a_2+\a_3\\
\sigma_2=\a_2+2\a_3+\a_4\\
\sigma_3=\a_4\\
\end{array}&

\begin{array}[t]{l|l|R|R|R|r|}
&&D_1&D_3&D_4^+&D_4^-\\
\hline
1&\sigma_1&1&0&-1&0\\
1&\sigma_2&-1&1&0&0\\
1&\sigma_3&0&-1&1&1\\
\end{array}\\

\hline

\sG_2&
\begin{array}[t]{l}
\bild{\FarbeD\Grechts}\\
\tau=2\a_1+\a_2\\
\end{array}\\

(1)&
\begin{array}[t]{l}
\bild{\Alinks\FarbeU\FarbeO\Grechts\FarbeD}\\
\sigma_1=\a_1\\
\sigma_2=\a_1+\a_2\\
\end{array}&

\begin{array}[t]{l|l|r|R|R|}
&&D_1^+&D_1^-&D_2\\
\hline
1&\sigma_1&1&1&-1\\
1&\sigma_2&0&-1&1\\
\end{array}\\

\hline

\sG'_2&
\begin{array}[t]{l}
\bild{\Alinks\FarbeM\Grechts\FarbeD}\\
\tau=\a_1+\a_2\\
\end{array}\\

(1)&
\begin{array}[t]{l}
\bild{\Alinks\FarbeU\FarbeO\Grechts\FarbeU\FarbeO}\\
\sigma_1=\a_1\\
\sigma_2=\a_2\\
\end{array}&

\begin{array}[t]{l|l|r|R|R|r|}
&&D_1^+&D_1^-&D_2^+&D_2^-\\
\hline
1&\sigma_1&1&1&-1&0\\
1&\sigma_2&-2&-1&1&1\\
\end{array}\\

\hline

\sD_n&
\begin{array}[t]{l}
\bild{\FarbeD\Arechts\ddd\Drechts}\\
\tau=2\a_1\+2\a_{n-2}+\a_{n-1}+\a_n\\
\end{array}\\

(1)&
\begin{array}[t]{l}
\bild{\Alinks\WurzelAn\IndexUnten{\a_\nu\rlap{\ (\nu=1,\ldots,n-2)}}\Arechts\FarbeD\Arechts\ddd\Drechts}\\
\sigma_1=\a_1\+\a_\nu\\
\sigma_2=2\a_{\nu+1}\+2\a_{n-2}+\a_{n-1}+\a_n\\
\end{array}&

\begin{array}[t]{l|l|r|R|R|}
&&D_1&D_\nu&D_{\nu+1}\\
\hline
2&\sigma_1&1&1&-1\\
1&\sigma_2&0&-2&2\\
\end{array}\\

(2)&
\begin{array}[t]{l}
\bild{\Alinks\FarbeM\Arechts\ddd\Arechts\Drechts
\FarbeMex(4.5,0.5)\FarbeMex(4.5,-0.5)

\draw[thick](0,0)--++(0.1,0.1);
\draw[thick](0,0)--++(0.1,-0.1);

\draw[thick,rotate around ={45:(4.5,0.5)}]
(4.5,0.5)--++(-0.1,0.1)
foreach \x in {1,...,5} {--++(-0.05,0.05)--++(-0.05,-0.05)}
--++(-0.05,0.05);

\draw[thick]
($(0,0)+(0.1,0.1)$)
foreach \x in {1,...,38} {--++(0.05,0.05)--++(0.05,-0.05)}
--++(0.06,0.06);

\draw[thick]
($(0,0)+(0.1,-0.1)$)
foreach \x in {1,...,38} {--++(0.05,-0.05)--++(0.05,0.05)}
--++(0.06,-0.06);

\draw[thick,rotate around ={135:(4.5,-0.5)}]
(4.5,-0.5)--++(0.1,0.1)
foreach \x in {1,...,5} {--++(0.05,0.05)--++(0.05,-0.05)}
--++(0.05,0.05);}\\
\sigma_1=\a_1\+\a_{n-2}+\a_{n-1}\\
\sigma_2=\a_1\+\a_{n-2}+\a_n\\
\end{array}&

\begin{array}[t]{l|l|r|R|R|}
&&D_1&D_{n-1}&D_n\\
\hline
1&\sigma_1&1&1&-1\\
1&\sigma_2&1&-1&1\\
\end{array}\\

\hline

\sB''_3&
\begin{array}[t]{l}
\bild{\Alinks\Arechts\Brechts\FarbeD}\\
\tau=\a_1+2\a_2+3\a_3\\
\end{array}\\

(1)&
\begin{array}[t]{l}
\bild{\Alinks\FarbeM\zz\Arechts\FarbeD\Brechts\FarbeU\FarbeO}\\
\sigma_1=\a_1+\a_2\\
\sigma_2=\a_2+a_3\\
\sigma_3=\a_3\\
\end{array}&

\begin{array}[t]{l|l|R|R|R|r|}
&&D_1&D_2&D_3^+&D_3^-\\
\hline
1&\sigma_1&1&1&-2&0\\
1&\sigma_2&-1&1&0&0\\
2&\sigma_3&0&-1&1&1\\
\end{array}\\

(2)&
\begin{array}[t]{l}
\bild{\Alinks\FarbeU\FarbeO\draw (0,0.57)--++(0,0.2)--++(1,0)--++(0,-0.2);\draw (0,-0.57)--++(0,-0.2)--++(2,0)--++(0,0.2);\Arechts\FarbeU\FarbeO\Brechts\FarbeU\FarbeO}\\
\sigma_1=\a_1\\
\sigma_2=\a_2\\
\sigma_3=\a_3\\
\end{array}&

\begin{array}[t]{l|l|R|R|R|r|}
&&D_1^+&D_1^-&D_2^-&D_3^+\\
\hline
1&\sigma_1&1&1&-2&-1\\
2&\sigma_2&1&-2&1&0\\
3&\sigma_3&-1&1&0&1\\
\end{array}\\

\hline

\end{longtable}

\end{document}